%% file: limit.tex
\documentclass[11pt]{amsart}
\usepackage{amsmath}

\usepackage{amssymb}
\usepackage{fancyhdr}
\usepackage{amsthm}
\usepackage{float}
\usepackage{graphicx}

\usepackage{enumerate}
\usepackage{accents}
\usepackage{url}

\input{defs}
\begin{document}

\title{Limits of sequences of pseudo-Anosov maps and of hyperbolic
3-manifolds}
\date{March 2020}
\author{Sylvain Bonnot}
\address{Departamento de Matem\'atica\\ IME-USP\\ Rua Do Mat\~ao
    1010\\ Cidade Universit\'aria\\ 05508-090 S\~ao Paulo SP\\ Brazil}
\email{sylvain@ime.usp.br}
\author{Andr\'e de Carvalho}
\address{Departamento de Matem\'atica Aplicada\\ IME-USP\\ Rua Do Mat\~ao
    1010\\ Cidade Universit\'aria\\ 05508-090 S\~ao Paulo SP\\ Brazil}
\email{andre@ime.usp.br}
\author{Juan Gonz\'alez-Meneses}
\address{Departamento de \'Algebra\\Instituto de Matem\'aticas (IMUS)\\Universidad de Sevilla\\Av. Reina Mercedes s/n\\41012 Sevilla\\Spain}
\email{meneses@us.es}
\author{Toby Hall}
\address{Department of Mathematical Sciences\\ University of Liverpool\\
    Liverpool L69 7ZL, UK}
\email{tobyhall@liverpool.ac.uk}

\begin{abstract}
  There are two objects naturally associated with a braid $\beta\in B_n$ of
  pseudo-Anosov type: a (relative) pseudo-Anosov homeomorphism
  $\varphi_\beta\colon S^2\to S^2$; and the finite volume complete hyperbolic
  structure on the 3-manifold $M_\beta$ obtained by excising the braid closure
  of~$\beta$, together with its braid axis, from~$S^3$. We show the disconnect
  between these objects, by exhibiting a family of braids
  $\{\beta_q:q\in\Q\cap(0,1/3]\}$ with the properties that: on the one hand,
  there is a fixed homeomorphism $\varphi_0\colon S^2\to S^2$ to which the
  (suitably normalized) homeomorphisms $\varphi_{\beta_{q}}$ converge as $q\to
  0$; while on the other hand, there are infinitely many distinct hyperbolic
  3-manifolds which arise as geometric limits of the form $\lim_{k\to\infty}
  M_{\beta_{q_k}}$, for sequences $q_k\to 0$.
\end{abstract}

\maketitle

\section{Introduction}
\label{sec:intro}

This article presents a somewhat surprising phenomenon on the interface between
the theories of surface homeomorphisms and of 3-manifold geometry. Two theorems
due to Thurston associate to certain mapping classes on a surface --- the
pseudo-Anosov mapping classes --- two different types of canonical objects.

\begin{itemize}
  \item The Classification Theorem for Surface
  Homeomorphisms~\cite{CTSH,FLP,CB} states that every irreducible mapping
  class which is not of finite order contains a pseudo-Anosov homeomorphism,
  which is unique up to topological conjugacy. Such a mapping class
  is said to be of pseudo-Anosov type.

  \item The Hyperbolization Theorem for Fibered
  3-Manifolds~\cite{HTF3M,Otal,McMullen} states that the mapping torus of a
  mapping class admits a complete hyperbolic metric of finite volume (unique up
  to isometry by the Mostow-Prasad Theorem) if and only if the mapping class is
  of pseudo-Anosov type.
\end{itemize}

In this paper we consider mapping classes of marked spheres,
represented by elements of Artin's braid groups: an $n$-braid $\beta\in B_n$
defines a mapping class on the $n$-marked disk, and hence on the
$(n+1)$-marked sphere. We say that~$\beta$ is of pseudo-Anosov type if and
only if the corresponding mapping class is, and in this case we can associate
to it:
\begin{itemize}
  \item a homeomorphism $\varphi_\beta\colon S^2\to S^2$, unique up to
  conjugacy, which is pseudo-Anosov relative to the marked points (that is,
  whose invariant foliations are permitted to have $1$-pronged singularities at
  these points); and

  \item the hyperbolic $3$-manifold\footnote{All of the $3$-manifolds in this paper are of the form $M_\beta$
  for some pseudo-Anosov braid~$\beta$, and we consider them equipped with
  their unique hyperbolic structures without further comment.} $M_\beta =
  S^3\setminus(\bbeta \cup A)$ --- where $\bbeta$ is the closure of $\beta$ and
  $A$ is its braid axis
  --- which is homeomorphic to the mapping torus of $\varphi_\beta$ (acting on
  the sphere punctured at the $n+1$ marked points).
\end{itemize}

We will present a family of pseudo-Anosov braids
$\{\beta_q:q\in\Q\cap(0,1/3]\}$, with $\beta_{m/n}\in B_{n+2}$, with the
following properties:
\begin{itemize}
  \item The pseudo-Anosov homeomorphisms $\varphi_q := \varphi_{\beta_q}:
  S^2\to S^2$ can be normalized in such a way that $\varphi_q\to\varphi_0$ as
  $q\to 0$, where $\varphi_0$ is a fixed sphere homeomorphism (the \emph{tight
  horseshoe map}, derived from Smale's horseshoe map).

  \item The hyperbolic $3$-manifolds $M_q:=M_{\beta_q}$ have the property that
  there are infinitely many distinct finite volume hyperbolic $3$-manifolds
  which can be obtained as geometric limits $\lim_{k\to\infty}M_{q_k}$ for some
  sequence $q_k\to 0$.
\end{itemize}

The braids $\beta_q$ are the \emph{NBT braids} of~\cite{HS}: they are
pseudo-Anosov braids for which the corresponding pseudo-Anosov homeomorphisms
$\varphi_q$ have particularly simple train tracks (see Remark~\ref{rmk:nbt}).
The fact that $\varphi_q\to\varphi_0$ as $q\to 0$ is a straightforward
consequence of results of~\cite{prime}: the main content of this paper is an
analysis of possible geometric limits of sequences $M_{q_k}$.

\medskip

It is interesting to contrast this work with the surprising discovery due to
Farb, Leininger, and Margalit~\cite{FLM} (see also~\cite{Agol}) of a universal
finiteness phenomenon for the mapping tori of small dilatation pseudo-Anosov
homeomorphisms: all such mapping tori can be obtained by Dehn surgery on a
finite collection of hyperbolic 3-manifolds. More precisely, given a constant
$P>0$, a pseudo-Anosov homeomorphism $\varphi\colon S\to S$ of a surface~$S$,
with dilatation~$\lambda$, is said to have small dilatation if $|\chi(S)|
\log\lambda \le P$. It follows from a result of Penner~\cite{Penner} that, for
sufficiently large~$P$, the set of small dilatation pseudo-Anosovs (as~$S$
ranges over all surfaces of negative Euler characteristic) is infinite.
Nevertheless, it is shown in~\cite{FLM} that, after puncturing at the
singularities of the invariant foliations of each pseudo-Anosov, there is only
a finite number of mapping tori associated with these maps.

Here, on the other hand, we consider sequences of pseudo-Anosov homeomorphisms
of the punctured sphere (punctured, in fact, exactly at the singularities of
the invariant foliations, although these include $1$-pronged singularities),
all of which converge to the same sphere homeomorphism, and show that the
corresponding sequences of mapping tori have infinitely many distinct geometric
limits. Since our sequences of pseudo-Anosovs have dilatations converging
to~$2$, and are defined on punctured spheres with unbounded Euler
characteristics, they do not have small dilatations.

\medskip

The principal technique used in the paper is Dehn surgery, and we now briefly
recap some key definitions and results, in order to fix conventions (which are
taken from section~9 of Rolfsen's book~\cite{rolfsen}). Let~$L =
L_1\cup\dots\cup L_n$ be a link in~$S^3$ with components~$L_i$, and let~$N$ be
a closed tubular neighborhood of~$L_1$ which is disjoint from the other
components of~$L$. Pick a basis $([\mu],[\lambda])$ for $H_1(\partial N, \Z)$
such that the `meridian' $\mu$ is contractible in~$N$ and the
`longitude'~$\lambda$ has linking number~$0$ with~$L_1$.
 
If~$J$ is a homotopically non-trivial simple closed curve in~$\partial N$, then
we can construct a $3$-manifold
\[
  M = \left(S^3 \setminus (L \cup \accentset{\circ}{N}) \right)\cup_h N,
\]
where $h\colon\partial N\to\partial N$ is a homeomorphism which takes $\mu$
onto~$J$. Writing $[J] = b[\mu] + a[\lambda] $, we say that $M$ is obtained
from $S^3\setminus L$ by \emph{Dehn filling $L_1$ with surgery coefficient
$r=b/a$}: this definition is independent of the choices of orientations of
$\mu$, $\lambda$ and~$J$. (This corresponds to Dehn filling coefficient $(b,a)$
in the notation used by SnapPy~\cite{SnapPy}, where the coefficients $(b,a)$
and $(-b,-a)$ lead to the same surgery. We will always assume that $a$ and $b$
are coprime.) We define the surgery coefficient~$r$ to be~$\infty$ if and only
if $[J]=\pm[\mu]$ (so that $b=1$ and $a=0$). In this case $M=S^3\setminus
(L_2\cup\dots\cup L_n)$: that is, filling~$L_1$ with surgery
coefficient~$\infty$ is the same as erasing the component~$L_1$ from the
link~$L$.
 
Suppose now that we have assigned surgery coefficients to some of the
components of~$L$, and that~$L_1$ is an unknotted component of~$L$. Applying a
positive meridional twist to the (solid torus) complement of a tubular
neighborhood of~$L_1$ is referred to as \emph{performing a $+1$ twist
on~$L_1$}: if~$D$ is a disk bounded by~$L_1$ which the other components of~$L$
intersect transversely, then the effect of this twist on the link~$L$ is to
replace each segment of~$L$ which intersects~$D$ with a helix which screws
through a collar of~$D$ in the right-handed sense. If $t\in\Z$, then
\emph{performing a $t$ twist on~$L_1$} means performing~$t$ such twists if
$t\ge0$, or $-t$ left-handed twists if $t<0$.
 
The revised link $L'$ after a $t$ twist on~$L_1$ describes the same
$3$-manifold as~$L$ provided that the surgery coefficients (on those components
of~$L$ which have them) are updated using the formul\ae:
\begin{equation}
\label{eq:update}
  \begin{aligned}
    r_1(L_1) &= \frac{1}{t + 1/r_0(L_1)}, \\
    r_1(L_i) &= r_0(L_i) + t(\lk(L_1, L_i))^2 \qquad(i>1),
  \end{aligned}
\end{equation}
where $r_0(L_i)$ and $r_1(L_i)$ are the surgery coefficients on~$L_i$ before
and after the twist, and $\lk(L_1, L_i)$ is the linking number of $L_1$ with
$L_i$.

In this paper, we will only perform twists in the case where~$L=\bbeta\cup A$
is the closure of a braid together with its axis; and we will only perform them
on either the braid axis~$A$ or a fixed component of $\bbeta$ (one which
corresponds to a single string of the braid). It will therefore be convenient
to describe the effects of such twists directly on the braid.
\begin{enumerate}[(a)]
  \item A $t$-twist on the braid axis~$A$ replaces $\beta$ with
  $\beta\theta^{-t}$, where~$\theta$ is the full twist in the braid group.

  \item Figure~\ref{fig:twist-1}~(a) is a schematic
  representation of~$\bbeta\cup A$, where $\beta$ has a fixed string which
  links one of the other strings. The effect of a $-1$ twist on the
  corresponding component of $\bbeta$ is shown in (b), which is followed by
  conjugacy (exchanging the red and the first black string) to obtain the braid
  of (c). Because this braid has the same structure as~$\beta$, the process can
  be repeated $t-1$ more times to obtain the braid of (d), which is the effect
  of applying a $-t$ twist on the fixed string. It has $t$ more strands
  than~$\beta$.

  We shall also consider twists on fixed strings which link a ribbon of other
  parallel strings of the braid. Figure~\ref{fig:twist-2} shows the effect of a
  $-t$ twist in this case, determined analogously. If the ribbon consists
  of~$m$ strings, then this increases the number of strings of~$\beta$ by $tm$.

  In order to carry out a $+t$ twist on a fixed string, we will
  conjugate~$\beta$ to take the form of the right hand side of
  Figure~\ref{fig:twist-2}. The $+t$ twist will then reduce it to the braid on
  the left hand side.
\end{enumerate}

\begin{figure}[htbp]
  \begin{center}
    \includegraphics[width=16cm]{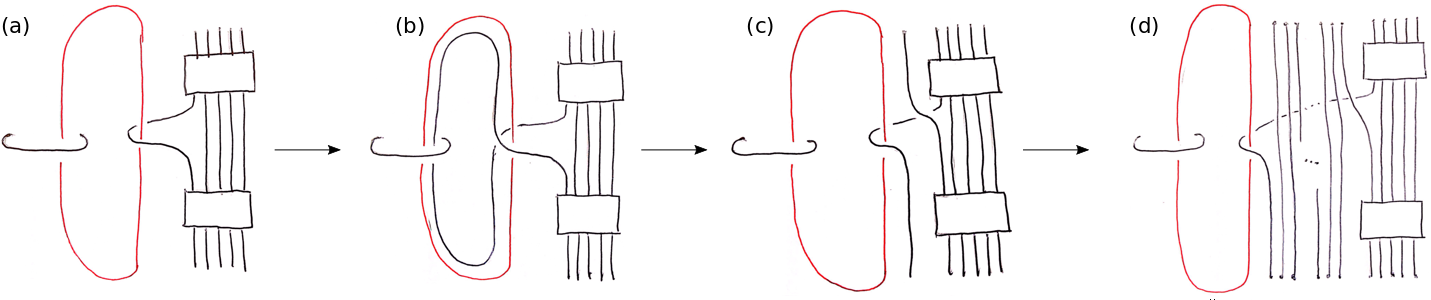}
    \caption{$-1$ and $-t$ twists on a fixed string which links one other string.}
    \label{fig:twist-1}
  \end{center}
\end{figure}

\begin{figure}[htbp]
  \begin{center}
    \includegraphics[width=8cm]{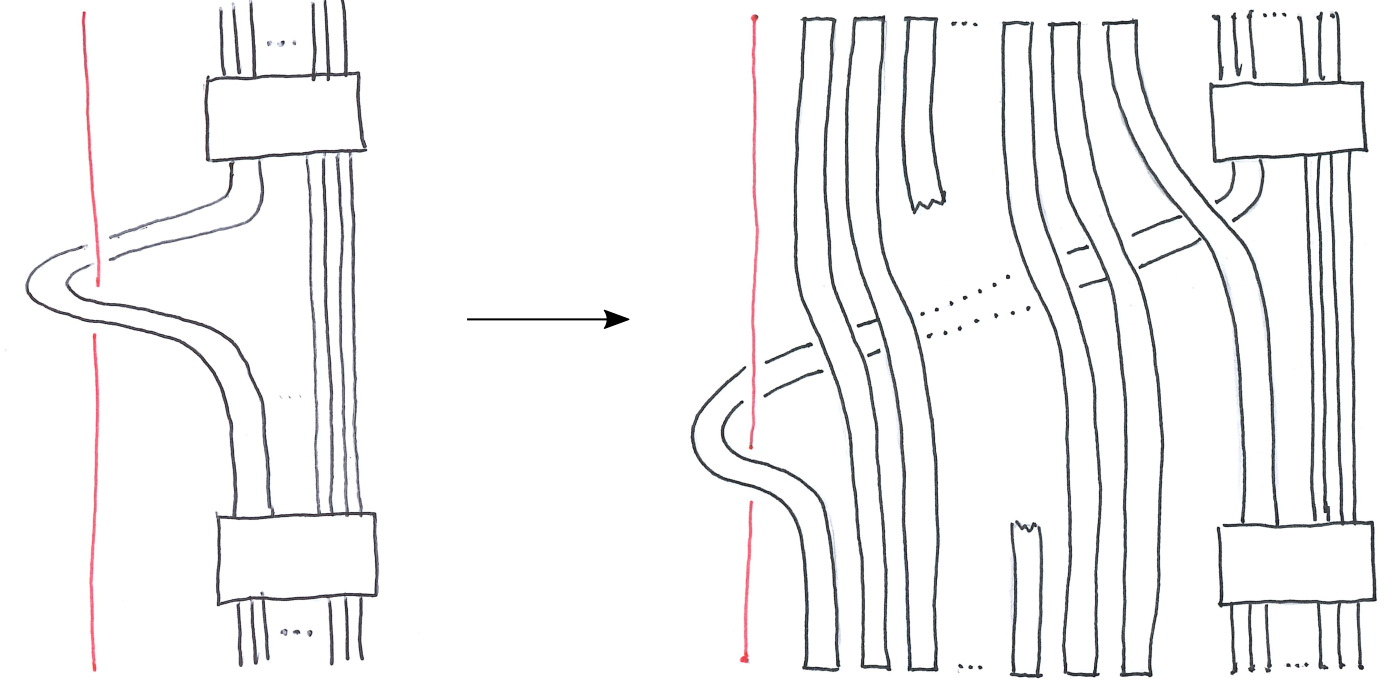}
    \caption{A $-t$ twist on a fixed string which links a ribbon of parallel
    strings.}
    \label{fig:twist-2}
  \end{center}
\end{figure}
 
We will use the following simplified version of Thurston's Hyperbolic Dehn
Surgery Theorem, which follows from Chapters~4 and~5 of~\cite{GT3M}, see
also~\cite{BP,NZ}. 

\begin{theorem}
\label{thm:HDST}
  Let $L = L_1\cup \dots \cup L_n$ be a link in $S^3$ such that $M:=S^3
  \setminus L$ is a complete hyperbolic 3-manifold of finite volume,
  $r_i=b_i/a_i$ be a sequence of rationals with $a_i^2+b_i^2\to\infty$, and
  $M_i$ be the sequence of 3-manifolds obtained by Dehn filling $L_1$ with
  surgery coefficients $r_i$. Then $M_i$ converges geometrically to $M$, and
  the convergence is non-trivial in the sense that $M_i$  and $M$ are distinct
  for all~$i$, so that there are infinitely many distinct 3-manifolds~$M_i$.
\end{theorem}

\section{The braids $\beta_q$}
\label{sec:beta_q}
Recall that the \emph{positive permutation braid}~$\beta\in B_n$ defined by a
permutation $\pi\in S_n$ is the unique $n$-braid which induces the
permutation~$\pi$ on its strings, and which has the properties that every pair
of strings crosses at most once, and that every crossing is in the positive
sense (we adopt the convention, following Birman~\cite{birman}, that a braid
crossing is positive if the left string crosses over the right one). Thus a
diagram of $\beta$ can be constructed by drawing the $1^\text{st}$ to the
$n^\text{th}$ strings in order, with the $i^\text{th}$ string going from
position~$i$ to position~$\pi(i)$ and passing underneath any intervening
strings which have already been drawn.

The following definition is from theorem~2.1 of~\cite{HS}, and the fact that
the braids defined are of pseudo-Anosov type is contained in the proof of
theorem~2.3 of the same paper. (There the braids~$\beta'_q$ are also defined
for $q\in(1/3, 1/2)$, but this is done in a different way and, since we are
only interested in limits as $q\to 0$, is not relevant here.) Here and
throughout the paper, when we write a positive rational number as $m/n$, we
will always assume that $m$ and $n$ are coprime and positive.
\begin{definition}[The braids $\beta'_q$]
\label{def:beta'}
  Let $q = m/n\in\Q\cap(0,1/3]$. The braid $\beta'_q\in B_{n+2}$ is the
  positive permutation braid (see Figure~\ref{fig:nbt}) defined by the cyclic
  permutation
  \begin{equation}
  \label{eqn:pi_q}
    \pi_q(r) = 
    \begin{cases}
      r + m         & \text{ if }1\le r\le n-3m+1,\\
      r + m + 1     & \text{ if }n-3m+2 \le r \le n-2m+1,\\
      2n - 2m + 4-r & \text{ if }n-2m+2 \le r \le n-m+1, \\
      n - 2m + 2   & \text{ if }r = n-m+2,\\
      n + 3 - r     & \text{ if }n-m+3 \le r \le n+2.
    \end{cases}
  \end{equation}
\end{definition}

It is helpful to organize the strings of $\beta'_q$ in \emph{ribbons} of
parallel strings: the~5 cases of~\eqref{eqn:pi_q} yield, in order:
\begin{itemize}
  \item A ribbon of width~$n-3m+1$ which moves~$m$ places to the right.

  \item A ribbon of width~$m$ which moves $m+1$ places to the right, thus leaving the target in position $n-2m+2$ unassigned.

  \item A ribbon of width~$m$ which is sent to the final~$m$ target positions with a half twist.

  \item A `rogue' string, which ends at the unassigned target in position
  $n-2m+2$.

  \item A ribbon of width~$m$, which is sent to the first~$m$ target positions with a half twist.
\end{itemize}

\begin{figure}[htbp]
  \begin{center}
    \includegraphics[width=7cm]{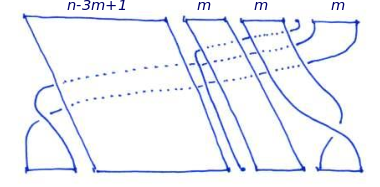}
    \caption{The braid $\beta'_{m/n}\in B_{n+2}$.}
    \label{fig:nbt}
  \end{center}
\end{figure}

\begin{definition}[The braids~$\beta_q$]
  It will be convenient for us to conjugate the braids $\beta'_q$ by a half
  twist of the final~$m$ strings, thereby turning the half twist on the final
  ribbon into a full twist, and removing the half twist on the penultimate
  ribbon: these conjugated braids will be denoted~$\beta_q$
  (Figure~\ref{fig:beta}). (The braids~$\beta_q$ can be seen as circular
  braids, as shown on the right of the figure, with each string other than the
  rogue one rotating around the circle by either~$m$ or~$m+1$ positions. This
  point of view motivates constructions later in the paper --- see
  Definitions~\ref{defn:gamma} and~\ref{defn:delta}.)
\end{definition}

\begin{figure}[htbp]
  \begin{center}
    \includegraphics[width=14cm]{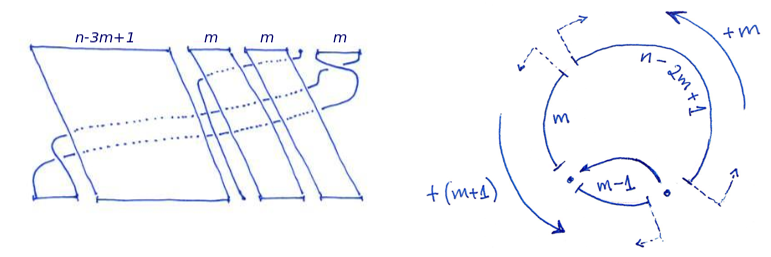}
    \caption{The braid $\beta_{m/n}\in B_{n+2}$, and as a circular braid.}
    \label{fig:beta}
  \end{center}
\end{figure}

\section{Pseudo-Anosov convergence to the tight horseshoe}
\label{sec:pA_converge}
The tight horseshoe map~\cite{extensions} $\varphi_0\colon S\to S$ is a
2-sphere homeomorphism which can be obtained by collapsing the horizontal and
vertical gaps in the invariant Cantor set of Smale's horseshoe
map~\cite{smale}. In order to define it directly, we start with its sphere~$S$
of definition, which is obtained by making identifications along the sides of a
unit square~$\Sigma$ as depicted in Figure~\ref{fig:tight-hs}. Infinitely many
segments along the boundary of~$\Sigma$, two of length $1/2^i$ for each $i\ge
0$, are folded in half (so that the points of each segment, other than the
center point, are identified in pairs). The top and right edges of~$\Sigma$ are
each a single folded segment, and the other segments are arranged on the left
and bottom sides in decreasing order of length from the top left and bottom
right vertices respectively. The fold segment endpoints, together with the
bottom left corner, are identified to a single point~$\infty$. It can be shown
(see for example~\cite{paper}) that the space~$S$ so obtained is a topological
sphere (and, in fact, that the Euclidean structure on~$\Sigma$ induces a well
defined conformal structure on~$S$).

\begin{figure}[htbp]
  \begin{center}
    \includegraphics[width=5cm]{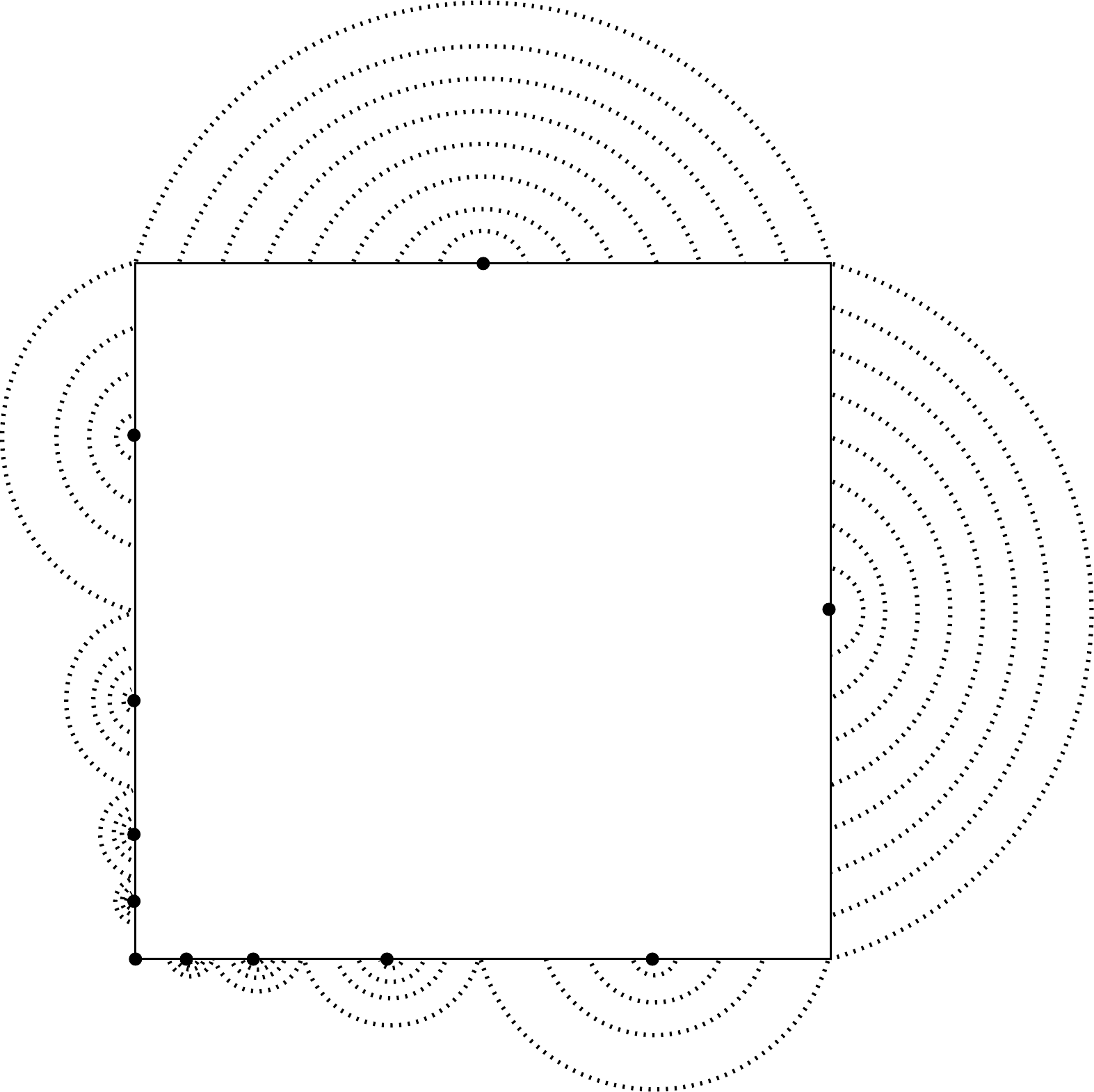}
    \caption{The sphere~$S$ of definition of the tight horseshoe map.}
    \label{fig:tight-hs}
  \end{center}
\end{figure}

To define the tight horseshoe map, let $F\colon\Sigma\to\Sigma$ be the
(discontinuous and non-injective) map defined by
\[
  F(x,y) = 
    \begin{cases}
      (2x,\, y/2)     & \text{ if }x \le 1/2,\\
      (2-2x,\, 1-y/2) & \text{ if } x > 1/2.
    \end{cases}
\]  
That is, $F$ stretches~$\Sigma$ by a factor 2 horizontally, contracts it by a
factor $1/2$ vertically, and maps its left half to its bottom half, and its
right half, with a flip, to its top half. The identifications on~$\Sigma$ are
precisely those needed to make~$F$ continuous and injective, so that it defines
a homeomorphism $\varphi_0\colon S\to S$, the tight horseshoe map. (It is an
example of a \emph{generalized pseudo-Anosov map}~\cite{gpa}: it has
(horizontal and vertical) unstable and stable invariant foliations, but these
foliations have infinitely many $1$-pronged singularities --- at the centers of
the fold segments --- accumulating on an `$\infty$-pronged singularity'
corresponding to the fold segment endpoints and the bottom left vertex.)

For each $q=m/n\in(0,1/3]\cap\Q$, let $\varphi_q\colon S^2\to S^2$ be a
pseudo-Anosov homeomorphism in the mapping class of the $(n+3)$-marked
sphere defined by $\beta_q$. The convergence of $\varphi_q$ to $\varphi_0$ as
$q\to 0$ is an immediate consequence of results from~\cite{prime}. The
following statement is a summary of the relevant parts of theorems~5.19
and~5.31 of that paper.

\begin{theorem}
\label{thm:prime}
  There is a continuously varying family $\{\chi_t\colon S^2\to S^2\}_{t\in (\sqrt2, 2]}$ of homeomorphisms of the standard 2-sphere, with the properties that:
  \begin{enumerate}[(a)]
    \item $\chi_2$ is topologically conjugate to $\varphi_0$; and
    \item there is a decreasing function $t\colon (0, 1/3]\cap\Q\to(\sqrt2,
    2)$, satisfying $t(q)\to 2$ as $q\to 0$, such that $\chi_{t(q)}$ is
    topologically conjugate to $\varphi_q$ for each~$q$.
  \end{enumerate}
\end{theorem}

\begin{remark}
\label{rmk:nbt}
  A brief discussion of the ideas surrounding Theorem~\ref{thm:prime} may be
  helpful to the reader. Boyland~\cite{Boy} defined the \emph{braid type} of a
  period~$n$ orbit~$P$ of an orientation-preserving disk homeomorphism $f\colon
  D^2\to D^2$ to be the isotopy class of $f\colon D^2\setminus P\to
  D^2\setminus P$, up to conjugacy in the mapping class group of the
  $n$-punctured disk: the braid type can therefore be described --- although
  not uniquely --- by a braid $\beta_P\in B_n$. He further defined the
  \emph{forcing relation}, a partial order on the set of braid types: one braid
  type forces another if every homeomorphism which has a periodic orbit of the
  former braid type also has one of the latter. The forcing relation therefore
  describes constraints on the order in which periodic orbits can appear in
  parameterized families of homeomorphisms.

  If $f$ is Smale's horseshoe map, then standard symbolic techniques associate
  a \emph{code} $c_P\in\{0,1\}^n$ to a period~$n$ orbit~$P$. This coding
  establishes a correspondence between the non-trivial periodic orbits of~$f$
  and those of the affine unimodal \emph{tent maps} $T_t\colon[0,1]\to[0,1]$
  defined for $t\in(1,2]$ by
  \[
    T_t(x) = \min(2+t(x-1), t(1-x)),
  \]     
  whose periodic orbits are likewise coded in a standard way. The
  braids~$\beta'_q$ of Definition~\ref{def:beta'} --- or, more accurately, the
  braids $\beta'_q$ for $q\in(0,1/2)\cap\mathbb{Q}$ alluded to before the
  definition --- are precisely the pseudo-Anosov braids describing braid types
  of horseshoe periodic orbits~$P_q$ which are \emph{quasi-one-dimensional}, in
  the sense that the braid types that they force are exactly those
  corresponding to the periodic orbits of the tent map $T_{t(q)}$ which has
  kneading sequence $c_{P_q}^\infty$~\cite{HS}.
      
  Another way to view the braids $\beta_{q}'$ is as the braids of horseshoe
  periodic orbits~$P_q$ whose mapping class is pseudo-Anosov and whose
  associated train tracks are the simplest possible: if the $1$-gons about the
  orbit points are ignored, then the union of the remaining edges is an arc.
  This means that the only singularities of the invariant foliations of
  $\varphi_q$ are 1-prongs at points of the orbit and an $n$-prong at $\infty$,
  where $q=m/n$. This is what makes the orbits $P_q$ quasi-one-dimensional: the
  induced map on the reduced train track (which is an interval) is a unimodal
  interval map.

  One way to construct the pseudo-Anosov map in a mapping class is as a factor
  of the natural extension of a corresponding train track map. In~\cite{prime},
  a similar method is used to construct a \emph{measurable pseudo-Anosov}
  homeomorphism from the natural extension of each tent map $T_t$ with
  $t>\sqrt{2}$: these form the continuously varying family $\chi_t$ of
  Theorem~\ref{thm:prime}. They are pseudo-Anosov maps if and only if the
  kneading sequence of the tent map is periodic and is the horseshoe code of
  one of the braids $\beta_{q}'$, i.e., if and only if  $t=t(q)$ for some $q\in
  (0,1/2)\cap\mathbb{Q}$, and in this case $\chi_{t(q)}$ is topologically
  conjugate to $\varphi_q$.

  Theorem~\ref{thm:prime} also provides limits of the pseudo-Anosov
  homeomorphisms $\varphi_q$ as $q$ tends to an irrational~$\xi$, or to a
  rational~$r$ either from above or from below (the image of~$t$ is discrete).
  All such limits are generalized pseudo-Anosov homeomorphisms.
\end{remark}

\section{Convergence of mapping tori}
\label{sec:hyp_converge}

Let $\nu=\ell/m\in [0,1)\cap\Q$, and consider the corresponding sequence
$(q^{(\nu)}_k)_{k\ge 3}$ of rationals defined by $q^{(\nu)}_k =
\frac{m}{km+\ell}$. By the description of the ribbon structure of the braids
$\beta_q$ in Section~\ref{sec:beta_q}, the braid~$\beta_{q^{(\nu)}_{k}}$ is as
depicted in Figure~\ref{fig:beta}, with the first ribbon having width
$(k-3)m+\ell+1$ and the others having width~$m$.

In this section we will show that, for each~$\nu$, the mapping tori
$M_{q^{(\nu)}_k}$ converge geometrically as $k\to\infty$ to a hyperbolic
manifold $\hatM_\nu$ of finite volume. In the following section, we will prove
that the set $\{\hatM_\nu\colon \nu\in[0,1)\cap\Q\}$ is infinite.

The crucial observation is that the sequence of mapping tori $M_{q^{(\nu)}_k}$
can be obtained from a single finite-volume hyperbolic 3-manifold $\hatM_\nu$
by Dehn filling one of its cusps with a sequence of distinct surgery
coefficients $r_k$: it therefore follows from Theorem~\ref{thm:HDST} that the
sequence of mapping tori converges geometrically to $\hatM_\nu$.

The manifolds $\hatM_\nu$ are themselves mapping tori, corresponding to braids
$\gamma_\nu$ which are obtained from $\beta_{q^{(\nu)}_3} =
\beta_{m/(3m+\ell)}$ by adding one additional string on the left. This
additional string is chosen precisely in order that $\hatM_\nu$ is the
geometric limit of the sequence $M_{q^{(\nu)}_k}$ (see the proof of
Theorem~\ref{thm:fill-Mnu}).

\begin{definition}[The braids~$\gamma_\nu$] 
\label{defn:gamma}
  Let $\nu = \ell/m\in[0,1) \cap \Q$.
  The braid $\gamma_\nu\in B_{3m+\ell+3}$ is obtained from
  $\beta_{m/(3m+\ell)}$ by adding a fixed string on the left, which links with
  the final width~$m$ ribbon of $\beta_{m/(3m+\ell)}$ but not with the other
  strings, as depicted in Figure~\ref{fig:gamma}. (In the circular
  representation of Figure~\ref{fig:beta}, this corresponds to adding a fixed
  string, not linking the rogue string, through the center of the circle.)
\end{definition}

\begin{figure}[htbp]
  \begin{center}
    \includegraphics[width=7cm]{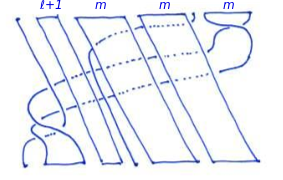}
    \caption{The braid $\gamma_{\ell/m}$.}
    \label{fig:gamma}
  \end{center}
\end{figure}

That $\gamma_\nu$ is a pseudo-Anosov braid follows from the fact that
$\beta_{m/(3m+\ell)}$ is. (Any reducing curve~$C$ would bound a disk~$D$
containing at least two but not all of the punctures associated with the
strings of~$\gamma_\nu$. $D$ cannot contain the puncture associated to the
fixed string, since then its image would also contain that puncture but a
different set of the other punctures; it cannot contain a proper subset of the
other punctures, since then $\beta_{m/(3m+\ell)}$ would be reducible; and it
cannot contain all of the other punctures since the associated strings link
with the fixed string.) Therefore $\hatM_\nu := S^3\setminus(\bgamma_\nu
\cup A)$ (where~$A$ is the braid axis) is a finite volume hyperbolic
$3$-manifold with~$3$ cusps.

\begin{theorem} 
\label{thm:fill-Mnu}
  Let $\nu\in[0,1)\cap\Q$ and $k\ge 1$. Dehn filling the cusp of
  $\hatM_\nu$ corresponding to the fixed string of~$\gamma_\nu$ with surgery
  coefficient~$1/k$ yields $M_{q^{(\nu)}_{k+3}}$.
\end{theorem}
\begin{proof}
  It is immediate from Figure~\ref{fig:twist-2} that performing a $-k$ twist on
  the component~$R$ of $\bgamma_\nu\cup A$ corresponding to the fixed string
  increases the width of the first ribbon of $\gamma_\nu$ from $\ell+1$ to
  $km+\ell+1$. By~\eqref{eq:update}, this changes the surgery coefficient
  on~$R$ to $r_1(R) = 1/(-k + 1/(1/k)) = \infty$, so that it can be erased,
  yielding the closure of the braid $\beta_{m/((k+3)m + \ell)} =
  \beta_{q_{k+3}^{(\nu)}}$ (see Figure~\ref{fig:beta}). That is, Dehn
  filling~$R$ with surgery coefficient~$1/k$ yields $M_{q_{k+3}^{(\nu)}}$ as
  required.
\end{proof}

The following corollary is now immediate from Theorem~\ref{thm:HDST}.

\begin{corollary}
\label{cor:lim-Mnu}
  For each $\nu\in[0,1)\cap\Q$ the sequence $M_{q^{(\nu)}_k}$ converges
  geometrically to $\hatM_\nu$.
\end{corollary}

\section{Infinitely many limit manifolds}
\label{sec:infinitely_many}

Figure~\ref{fig:limit_vols} is a plot of the volumes of the limit
manifolds~$\hatM_\nu$ against~$\nu$, generated by SnapPy~\cite{SnapPy}. The
points in red are those for which~$\nu$ is of the form $\kk/(\kk+1)$. In this
section we show how all of the corresponding manifolds $\hatM_{\kk/(\kk+1)}$
can be obtained by Dehn filling a cusp of another hyperbolic $3$-manifold~$M$
with a sequence of distinct surgery coefficients, so that, again by
Theorem~\ref{thm:HDST}, there are infinitely many distinct limit manifolds
$\hatM_{\kk/(\kk+1)}$ (which converge geometrically to~$M$ as $\kk\to\infty$).

\begin{figure}[htbp]
  \begin{center}
    \includegraphics[width=9cm]{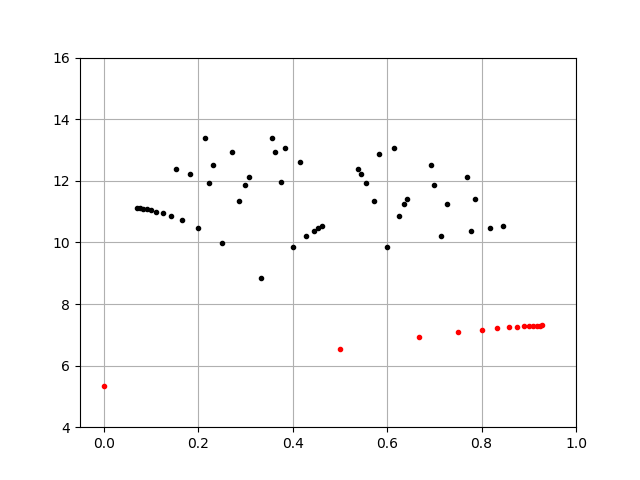}
    \caption{Volumes of the limit manifolds~$\hatM_\nu$ (for~$\nu$ with
    denominator $\le 14$).}
    \label{fig:limit_vols}
  \end{center}
\end{figure}

\medskip
\medskip
\medskip

\begin{remarks} \mbox{}
  \begin{enumerate}[(a)]
    \item Other apparently convergent sequences in Figure~\ref{fig:limit_vols}
    correspond to similar sequences $\nu_\kk$, such as $\nu_\kk=1/\kk$ and
    $\nu_\kk=\kk/(2\kk+1)$.

    \item The volume $5.333489\ldots$ of $\hatM_0$ suggests that it may be the
    magic manifold. To see that this is indeed the case, consider the braids
     depicted in Figure~\ref{fig:magic}, each representing the 3-manifold
     obtained by removing the braid closure together with its axis from~$S^3$.
     The braid on the left is $\gamma_{0/1}$, representing $\hatM_0$, while the
     one on the right represents the magic manifold (see for example Figure~3
     of~\cite{KT}). The operations converting each braid to the next are either
     twists on components of the associated links or braid conjugacies, and
     therefore leave the 3-manifolds unchanged. Specifically, these operations
     are, in order: conjugacy by $\sigma_4^{-1}$; a $+3$ twist on the red
     component (see Figure~\ref{fig:twist-1}~(d) and~(a)); conjugacy by
     $\sigma_2$; a $+1$ twist on the braid axis; and conjugacy by
     $\sigma_1\sigma_2$.
  \end{enumerate}

  \begin{figure}[htbp]
    \begin{center}
      \includegraphics[width=15cm]{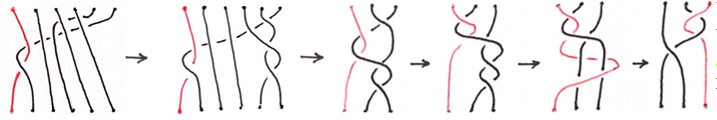}
      \caption{$\hatM_0$ is the magic manifold.}
      \label{fig:magic}
    \end{center}
  \end{figure}

\end{remarks}

The manifold~$M$ is obtained from the $10$-braid~$\delta$ of the following
definition (see Figure~\ref{fig:delta}), whose closure $\bdelta$ is a
three-component link. Note that the blue and green strings in the figure form a
braid conjugate to $\gamma_{0/1}$ (the conjugacy
$\sigma_1^{-1}\sigma_2^{-1}\sigma_3^{-1}\sigma_4^{-1}\sigma_5^{-1}$ moves the
fixed string from the left to the right of the braid diagram), and to this
braid has been added a $4$-string braid which `shadows' the blue strings. It is
not obvious a priori --- at least, not to the authors --- that Dehn filling the
`black' cusp of the resulting hyperbolic $3$-manifold should yield the
manifolds $\hatM_{\kk/(\kk+1)}$: rather, the braid~$\delta$ was found
experimentally using SnapPy~\cite{SnapPy}.

\begin{definition}
\label{defn:delta}
  Let $\delta=
  \sigma_6\sigma_5\sigma_4\sigma_3\sigma_9\sigma_8\sigma_8\sigma_9\sigma_7\sigma_6\sigma_5\sigma_4\sigma_3\sigma_2\sigma_1\sigma_8\sigma_7\sigma_6\sigma_5\sigma_4\sigma_3\sigma_2\sigma_1\sigma_8\sigma_6\in
  B_{10}$.
\end{definition}

 It can be checked, using the implementation~\cite{Trains} of the
Bestvina-Handel algorithm for train tracks of surface homeomorphisms~\cite{BH}
that $\delta$ is pseudo-Anosov, with its train track and image train track as
shown in Figure~\ref{fig:train_track}. The corresponding relative pseudo-Anosov
homeomorphism therefore has 1-pronged singularities at the marked points
corresponding to the blue and green strings of Figure~\ref{fig:delta}, a
$4$-pronged singularity at~$\infty$, and regular points at the black marked
points. Therefore~$M:=S^3\setminus(\bdelta \cup A)$ (where~$A$ is the braid
axis) is a finite volume hyperbolic 3-manifold with 4~cusps.

\begin{figure}[htbp]
  \begin{center}
    \includegraphics[width=14cm]{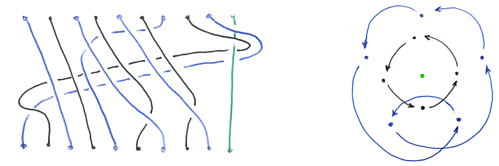}
    \caption{The braid $\delta\in B_{10}$.}
    \label{fig:delta}
  \end{center}
\end{figure}

\begin{figure}[htbp]
  \begin{center}
    \includegraphics[width=0.95\textwidth]{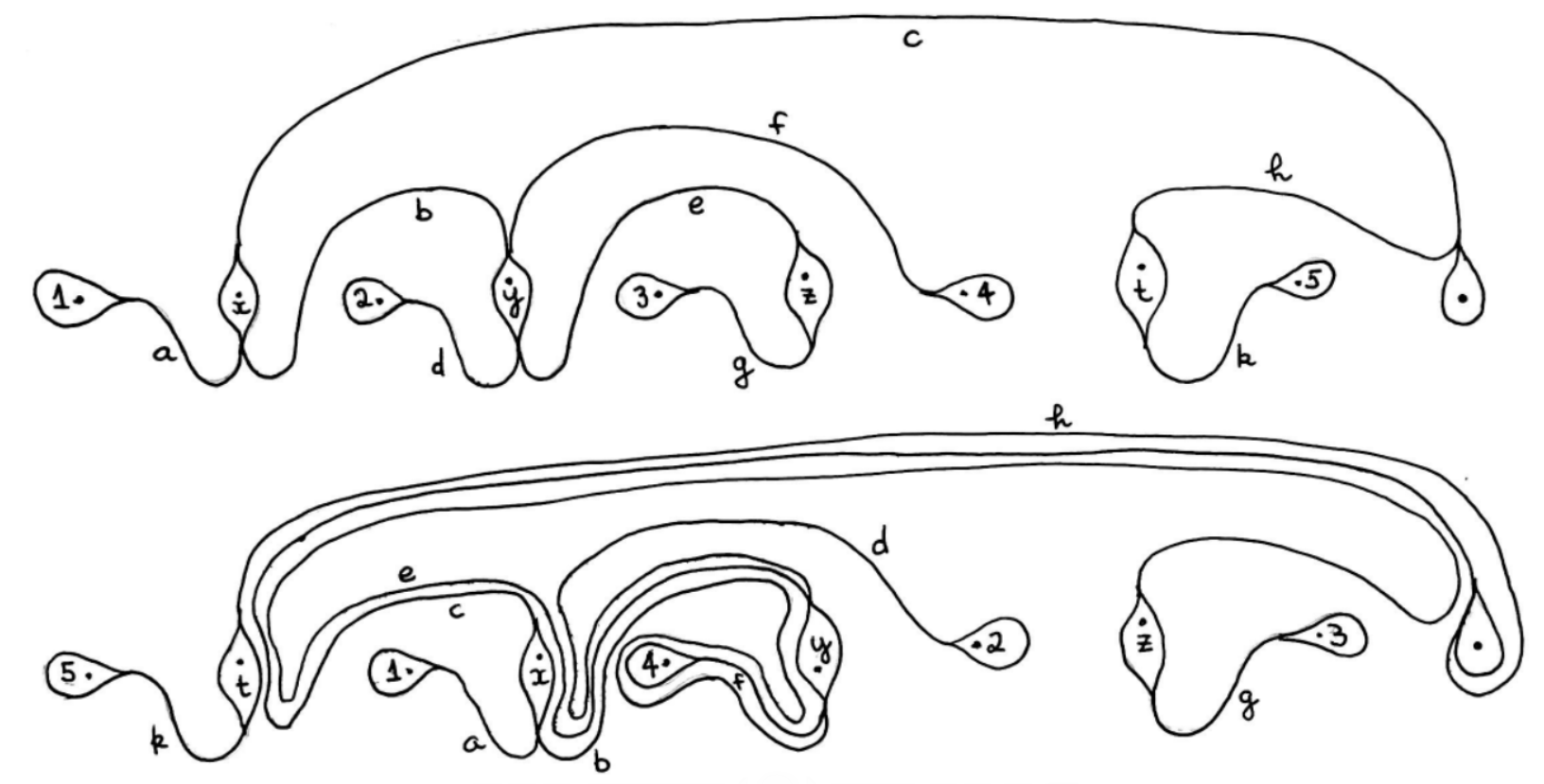}
    \caption{An invariant train track for $\delta$, and its image.}
    \label{fig:train_track}
  \end{center}
\end{figure}

\begin{theorem}
\label{thm:tdf}
  Let $\kk\ge 1$. Dehn filling the cusp of~$M$ corresponding to the black strings
  of Figure~\ref{fig:delta} with surgery coefficient $-4 + 1/\kk$ yields
  $\hatM_{\kk/(\kk+1)}$.
\end{theorem}

\begin{proof}
  The left hand side of Figure~\ref{fig:pre-conj} depicts a braid $\zeta$,
  which is $\delta$ together with an extra fixed string shown in red. We
  write~$B$ and~$R$ for the black and red components of $\bzeta$, which are
  unknotted. We need to show that filling~$B$ with coefficient $r_0(B)=-4+1/\kk$
  and~$R$ with coefficient $r_0(R)=\infty$ (i.e.\ erasing~$R$ from the link
  $\bzeta\cup A$) yields the $3$-manifold $\hatM_{\kk/(\kk+1)}$.

  \begin{figure}[htbp]
    \begin{center}
      \includegraphics[width=12cm]{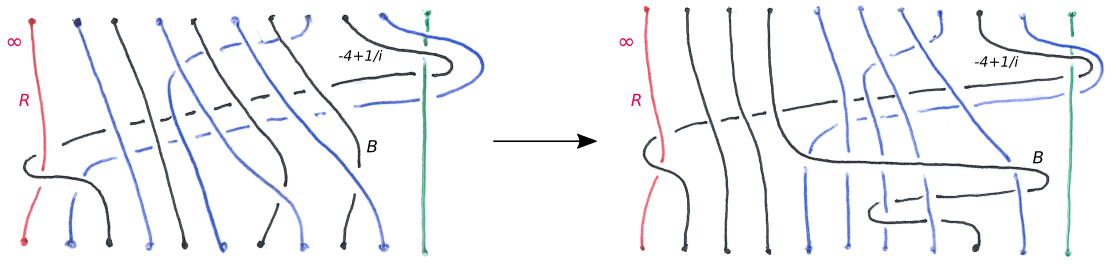}
      \caption{Conjugating $\zeta$ to prepare it for a $+3$ twist on~$R$.}
      \label{fig:pre-conj}
    \end{center}
  \end{figure}

  The braid on the right hand side of the figure is obtained by conjugating by
 $\sigma_7^{-2}\sigma_2^{-1}\sigma_4^{-1}\sigma_3^{-1}\sigma_6^{-1}\sigma_5^{-1}\sigma_4^{-1}$.
 Referring to Figure~\ref{fig:twist-1}, performing a~$+3$ twist on~$R$ yields
 the braid on the left of Figure~\ref{fig:tdf4-5}, and a conjugacy by
 $\sigma_6^{-1}\sigma_7$ gives the braid on the right hand side of the figure.
 By~\eqref{eq:update}, the updated surgery coefficients are:
  \begin{eqnarray*}
    r_1(R) &=& 1/(3+1/r_0(R)) = 1/3, \text{ and}\\ 
    r_1(B) &=& r_0(B) + 3 = -1 + 1/\kk,
  \end{eqnarray*}
  since $\lk(B, R) = 1$.

  \begin{figure}[htbp]
    \begin{center}
      \includegraphics[width=12cm]{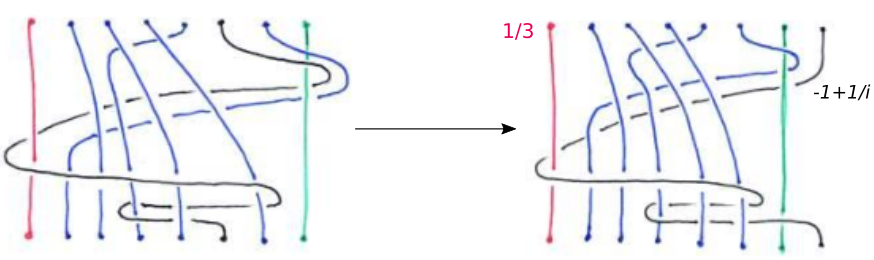}
      \caption{A $+3$ twist on~$R$, followed by a conjugacy.}
      \label{fig:tdf4-5}
    \end{center}
  \end{figure}

  Performing a $+1$ twist on the braid axis~$A$ yields the braid on the left
  hand side of Figure~\ref{fig:tdf6-7}, which a further conjugacy by
  $\sigma_7\sigma_6\sigma_5\sigma_4\sigma_3\sigma_2\sigma_1\sigma_1\sigma_2\sigma_3\sigma_4\sigma_5\sigma_6\sigma_7$
  --- to pull the black string around --- reduces to the right hand side of the
  figure. (Here and in Figure~\ref{fig:tdf8}, the parts of the blue strings
  which participate in the full twist have not been drawn, to clarify the
  diagrams.) The red component~$R$ and the black component~$B$ are now
  unlinked. The revised surgery coefficients are
  \begin{eqnarray*}
    r_2(R) &=& 1/3 + 1 = 4/3, \text{ and}\\ 
    r_2(B) &=& -1 + 1/\kk + 1 = 1/\kk,
  \end{eqnarray*}
  since $\lk(A, R)=\lk(A, B)=1$.

  \begin{figure}[htbp]
    \begin{center}
      \includegraphics[width=12cm]{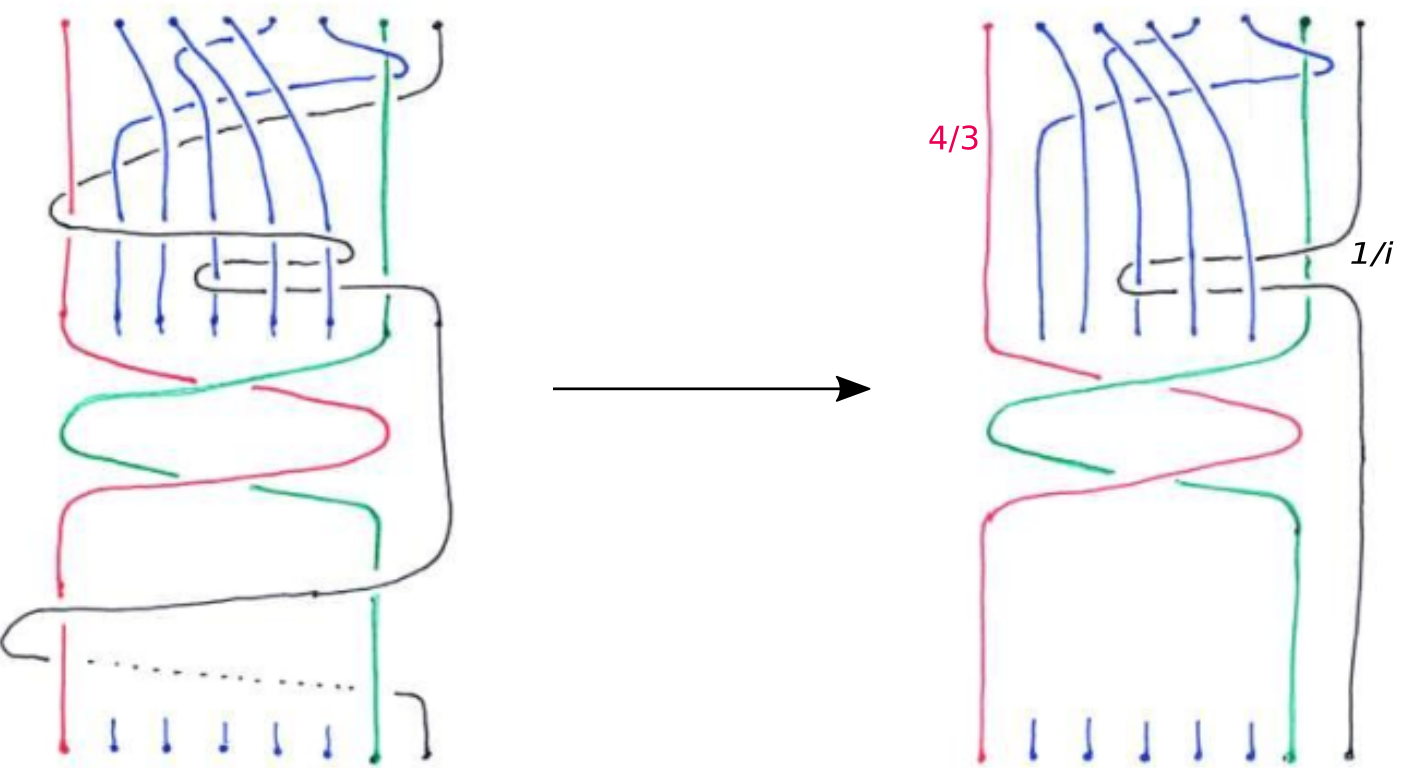}
      \caption{A $+1$ twist on the braid axis~$A$, followed by a conjugacy.}
      \label{fig:tdf6-7}
    \end{center}
  \end{figure}

  We can now carry out the surgery on~$B$. Performing a $-\kk$ twist on~$B$ yields the braid of Figure~\ref{fig:tdf8} (in which the ribbon contains~$\kk-1$ parallel strings). The surgery coefficient of~$B$ is
  \[
    r_3(B) = \frac{1}{-\kk + 1/(1/\kk)} = \infty,
  \]
  so that it can be removed (and is not shown in Figure~\ref{fig:tdf8}).
  Because~$R$ and~$B$ are unlinked, the surgery coefficient of~$R$ is
  unchanged: $r_3(R)=r_2(R)=4/3$.

  \begin{figure}[htbp]
    \begin{center}
      \includegraphics[width=4cm]{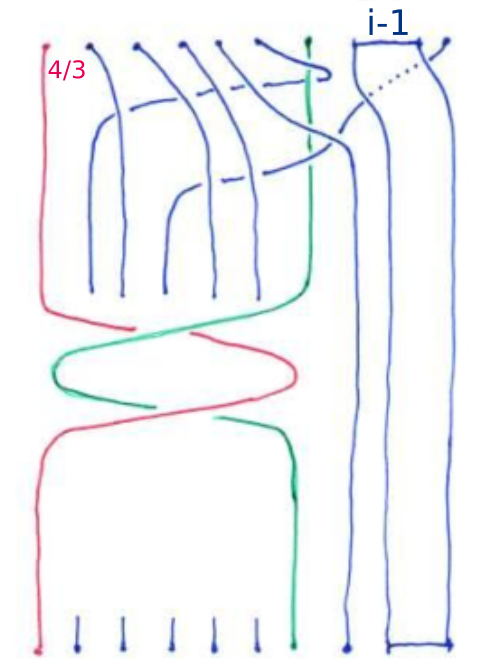}
      \caption{A $-\kk$ twist on~$B$ changes its surgery coefficient
      to~$\infty$.}
      \label{fig:tdf8}
    \end{center}
  \end{figure}

  We next perform a~$-1$ twist on~$A$, which produces the braid on the left
  hand side of Figure~\ref{fig:tdf9-10}, and changes the surgery coefficient
  of~$R$ to $r_4(R)=1/3$. A $-3$ twist on~$R$ therefore changes its coefficient
  to~$\infty$, so that it can be erased: this results in the braid on the right
  hand side of Figure~\ref{fig:tdf9-10}, in which each of the four ribbons
  contains~$\kk-1$ parallel strings.

  \begin{figure}[htbp]
    \begin{center}
      \includegraphics[width=15cm]{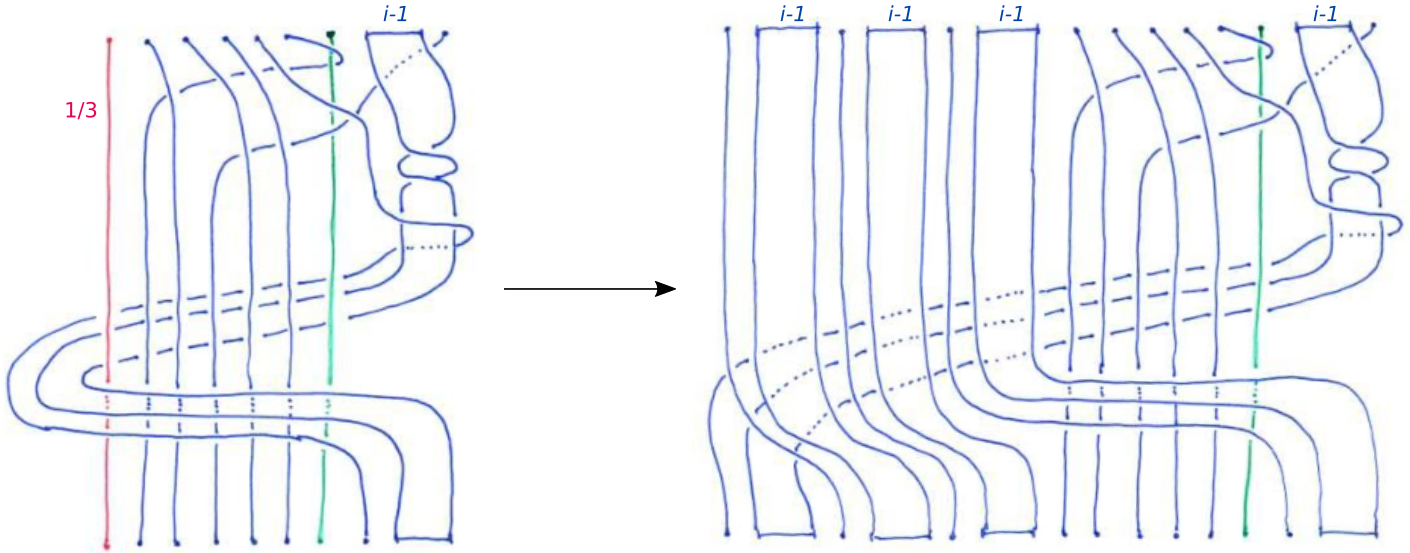}
      \caption{A $-1$ twist on~$A$ followed by a $-3$ twist on~$R$.}
      \label{fig:tdf9-10}
    \end{center}
  \end{figure}

  To complete the proof, we exhibit a braid conjugacy between the braid on the
  right hand side of Figure~\ref{fig:tdf9-10} and the braid $\gamma_{\kk/(\kk+1)}$
   --- that is, the braid of Figure~\ref{fig:gamma} with all four ribbons
       containing~$\kk+1$ parallel strings. (This conjugacy was discovered
       computationally, using sliding circuit set
       methods~\cite{cyclic1,cyclic2} for small values of~$\kk$ and
       extrapolating: the braids $\gamma_{\kk/(\kk+1)}$ have small sliding circuit
       sets but large ultra summit sets~\cite{gebhardt}.)  Two successive
       conjugacies are shown in Figure~\ref{fig:tdf11-12}. Here the first,
       second, and fourth ribbons have been enlarged by incorporating an
       additional parallel string, so that they each contain~$\kk$ parallel
       strings.

  \begin{figure}[htbp]
    \begin{center}
      \includegraphics[width=15cm]{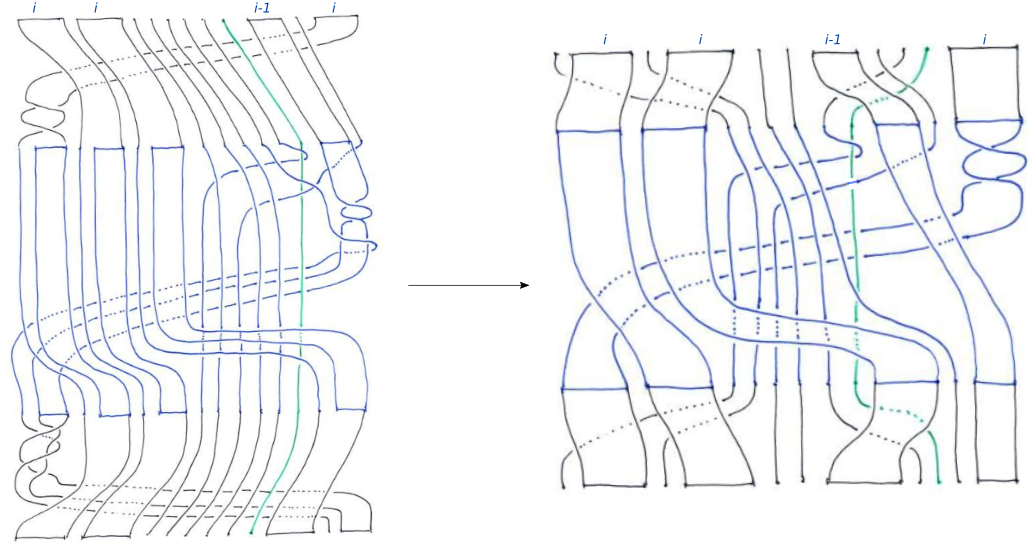}
      \caption{Successive conjugacies on the right hand side of
      Figure~\ref{fig:tdf9-10}.}
      \label{fig:tdf11-12}
    \end{center}
  \end{figure}

  Simplifying the braid on the right hand side of Figure~\ref{fig:tdf11-12} by
  isotopy of the strings yields the braid on the left hand side of
  Figure~\ref{fig:tdf13-14}. Again, we have incorporated additional parallel
  strings into ribbons, so that the first two ribbons contain~$\kk+1$ parallel
  strings, and the other two contain~$\kk$ parallel strings. A final conjugacy
  which moves the green string to the left, underneath all of the other
  strings, gives the braid on the right hand side of the figure, and
  incorporating additional parallel strings into the rightmost two ribbons
  yields $\gamma_{\kk/(\kk+1)}$ as required.

  \begin{figure}[htbp]
    \begin{center}
      \includegraphics[width=16cm]{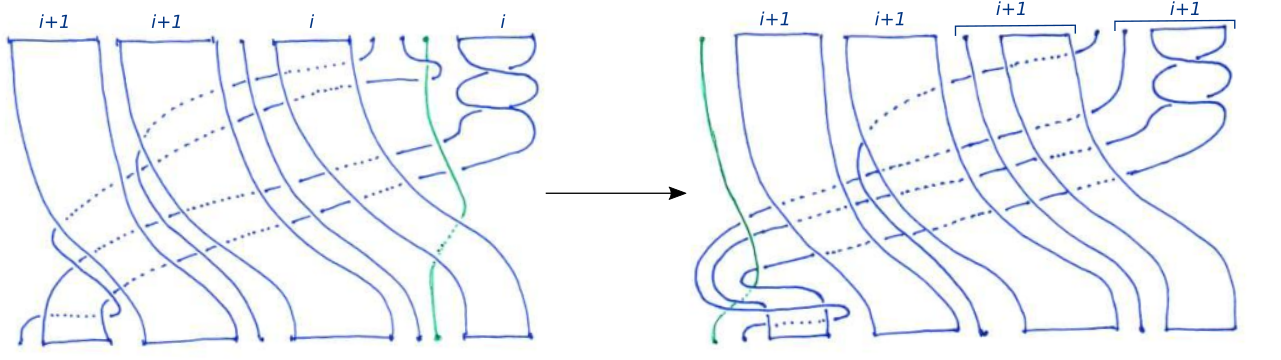}
      \caption{A simplified diagram of the right hand side of
      Figure~\ref{fig:tdf11-12}, followed by a conjugacy.}
      \label{fig:tdf13-14}
    \end{center}
  \end{figure}
\end{proof}

\begin{corollary}
  The sequence $\hatM_{\kk/(\kk+1)}$ converges geometrically to~$M$ as
  $\kk\to\infty$, and there are infinitely many distinct hyperbolic 3-manifolds
  $\hatM_{\kk/(\kk+1)}$.
\end{corollary}


\section{Acknowledgements}
The authors are grateful for the support of FAPESP grant 2016/25053-8 and CAPES
grant \mbox{88881.119100/2016-01}. AdC is partially supported by CNPq grant PQ
302392/2016-5. JGM is partially supported by Spanish Project
MTM2016-76453-C2-1-P and FEDER.

The authors appreciate the very helpful comments of an anonymous referee.
  
\bibliographystyle{amsplain}

\bibliography{limit}
\end{document}

%% file: defs.tex
\theoremstyle{plain}
  \newtheorem{theorem}{Theorem}

  \newtheorem{corollary}[theorem]{Corollary}

\theoremstyle{definition}
  \newtheorem{definition}[theorem]{Definition}
  
  \newtheorem{remark}[theorem]{Remark}
  \newtheorem{remarks}[theorem]{Remarks}

\newcommand{\Q}{{{\mathbb{Q}}}}
\newcommand{\Z}{{{\mathbb{Z}}}}

\newcommand{\lk}{\operatorname{lk}}

\newcommand{\bbeta}{{\overline{\beta}}}
\newcommand{\bgamma}{{\overline{\gamma}}}
\newcommand{\bdelta}{{\overline{\delta}}}
\newcommand{\bzeta}{{\overline{\zeta}}}
\newcommand{\hatM}{{\widehat{M}}}

\newcommand{\kk}{i}